\documentclass[12pt,a4paper,reqno]{amsart}
\usepackage[margin=1.2in]{geometry}
\usepackage[all]{xy}
\usepackage{amsmath}
\usepackage{amsfonts,amssymb}
\usepackage{hyperref}

\usepackage{booktabs}

\newtheorem{satz}{Satz}

\newtheorem{theorem}[satz]{Theorem}
\newtheorem*{theorem*}{Theorem}
\newtheorem{lemma}[satz]{Lemma}

\newtheorem{corollary}[satz]{Corollary}

\theoremstyle{definition}
\newtheorem{definition}[satz]{Definition}

\theoremstyle{remark}
\newtheorem*{remark*}{Remark}
\newtheorem{remark}[satz]{Remark}
\newtheorem*{claim*}{Claim}

\numberwithin{equation}{section}
\numberwithin{satz}{section}
\newcommand{\tensor}{\otimes}

\newcommand{\map}[1]{\stackrel{#1}{\longrightarrow}}

\newcommand{\un}[1]{\ensuremath{\protect\underline{#1}}}

\def\GL{\mathsf{GL}}

\def\SL{\mathsf{SL}}
\def\SU{\mathsf{SU}}
\def\U{\mathsf{U}}

\DeclareMathOperator{\Bun}{Bun}

\DeclareMathOperator{\Chain}{Chain}

\DeclareMathOperator{\Higgs}{Higgs}
\DeclareMathOperator{\Stab}{Stability}

\DeclareMathOperator{\cHom}{\mathcal{H}om}

\DeclareMathOperator{\Ext}{Ext}

\DeclareMathOperator{\rank}{rank}

\DeclareMathOperator{\supp}{supp}

\DeclareMathOperator{\rk}{rk}

\DeclareMathOperator{\Var}{Var}

\DeclareMathOperator{\sst}{ss}
\DeclareMathOperator{\wt}{wt}

\def\K0hat{\widehat{K}_0(\Var_k)}

\def\1halb{\frac{1}{2}}

\def\pprime{{\prime\prime}}

\newcommand{\xra}{\xrightarrow}

\newcommand{\norm}[1]{\lVert#1\rVert}

\def\sxymat{\xymatrix@C=1.5ex@R=0.8ex}
\def\grp{$\xymatrix{ R\times_{X}R  \ar[r]^-{\mu} & R \ar@<1ex>[r]^-{s}\ar@<-1ex>[r]_-{t} & X}$}
\def\dar{\ar@<-0.5ex>[r]\ar@<0.5ex>[r]}
\def\tar{\ar[r]\ar@<1ex>[r]\ar@<-1ex>[r]}
\newcommand{\dmap}[2]{\ar@<-0.5ex>[r]_-{#2}\ar@<0.5ex>[r]^-{#1}}
\newcommand{\dotarrow}[2]{\xymatrix{{#1}\ar@{..>}[r]&{#2}}}
\def\cart{\ar@{}[dr]|{\square}}

\def\cA{\mathcal{A}}

\def\cE{\mathcal{E}}
\def\cF{\mathcal{F}}
\def\cG{\mathcal{G}}

\def\cM{\mathcal{M}}

\def\cO{\mathcal{O}}

\def\cV{\mathcal{V}}
\def\cW{\mathcal{W}}

\def\bC{{\mathbb C}}

\def\bH{{\mathbb H}}

\def\bN{{\mathbb N}}

\def\bR{{\mathbb R}}

\def\bZ{{\mathbb Z}}

\begin{document}

\SelectTips{cm}{}
\dedicatory{To Nigel Hitchin}
\date{July 16, 2018}
\title[Chains and Higgs bundles]{Irreducibility of moduli of semistable Chains and applications to $\U(p,q)$-Higgs bundles}
\author[S. Bradlow]{Steven Bradlow}
\address{Department of Mathematics
	University of Illinois at Urbana-Champaign
	Office: 322 Illini Hall 
	Mailing Address: 1409 W. Green Street
	Urbana, IL 61801}
\email{bradlow@illinois.edu}
\author[O. Garc\'{\i}a-Prada]{Oscar Garc\'ia-Prada}
\address{Instituto de Ciencias Matem\'aticas CSIC-UAM-UC3M-UCM,
	calle Nicol\'as Cabrera, 15, Campus de Cantoblanco, 28049 Madrid, Spain}
\email{oscar.garcia-prada@icmat.es}
\author[P. Gothen]{Peter Gothen}
\address{Centro de Matem\'atica da Universidade do Porto,
	Faculdade de Ci\^encias da Universidade do Porto,
	Rua do Campo Alegre,
	4169-007 Porto, Portugal}
\email{pbgothen@fc.up.pt}
\author[J. Heinloth]{Jochen Heinloth}
\address{Universit\"at Duisburg--Essen, Fachbereich Mathematik, Universit\"atsstrasse 2, 45117 Essen, Germany}
\email{Jochen.Heinloth@uni-due.de}

\thanks{
  The authors acknow\-ledge support from U.S. National Science
  Foundation grants DMS 1107452, 1107263, 1107367 "RNMS: GEometric
  structures And Representation varieties" (the GEAR Network).
  Second author partially supported by the Spanish MINECO under the 
ICMAT Severo Ochoa grant No. SEV-2011-0087, and grant No. MTM2013-43963-P 
and by the European Commission Marie Curie IRSES  MODULI Programme 
PIRSES-GA-2013-612534.
  Third author partially supported by CMUP (UID/MAT/00144/2013) and the project
  PTDC/MAT-GEO/2823/2014 funded by FCT (Portugal)
  with national funds.
  Fourth author partially supported through SFB/TR 45 of the DFG
} 

\begin{abstract}
We give necessary and sufficient conditions for moduli spaces of semistable chains on a curve to be irreducible and non-empty. This gives information on the irreducible components of the nilpotent cone of $\GL_n$-Higgs bundles and the irreducible components of moduli of systems of Hodge bundles on curves. As we do not impose coprimality restrictions, we can apply this to prove connectedness for moduli spaces of $\U(p,q)$-Higgs bundles.
\end{abstract}

\maketitle

\section{Introduction}

The aim of this article is to show that moduli spaces of systems of Hodge
  bundles on curves and character varieties for
the unitary groups $\U(p,q)$ are connected, once the numerical
invariants of the corresponding bundles are fixed
(Theorems~\ref{thm:non-empty-irreducible-alpha-higgs} and
\ref{thm:upq-connected}). For the character varieties it has been
known for a long time
\cite{bradlow-garcia-prada-gothen:2004,bradlow-garcia-prada-gothen:2003}
that the closure of the space of representations corresponding to
stable Higgs bundles is connected, but (except for some special cases
\cite{xia:2003,markman-xia:2002}) the possibility of components without stable points remained a somewhat nagging possibility.

We approach the problem using the fact that systems of Hodge bundles can be
viewed as holomorphic chains. As a consequence we also show that the
necessary conditions found in \cite{GPH} for the existence of
semistable chains on a curve are also sufficient conditions. This also
gives a concrete, but complicated, enumeration of the irreducible
components of the so called global nilpotent cone on a curve in case rank and degree are coprime, and an estimate in the general case (Corollary \ref{cor:irredcomp}). 

The results follow from a more detailed study of wall crossing phenomena for moduli spaces of semistable chains for varying stability conditions that extend arguments from \cite{H}.
The key problem for our applications is to extend these arguments to the boundary case of walls containing the stability condition induced from sta\-bi\-li\-ty of Higgs bundles.
There the available results turn out to be just strong enough to show that an irreducible stack for a slightly larger stability parameter contains representative of every S-equivalence class.

For the application to $\U(p,q)$-Higgs bundles it  is sufficient to prove the above results for moduli spaces of triples only, for which the wall crossing is easier to understand. We have therefore tried to illustrate our arguments in this special case throughout, before giving the general argument. 

The structure of the article is as follows. In Section \ref{sec:setup}
we set up the notation and recall the standard notions of stability
for holomorphic chains. In
Section \ref{sec:sufficient} we prove that the necessary conditions
for the existence of semistable chains found in \cite{GPH} are also
sufficient if the stability parameter $\alpha$ is larger than the value $\alpha_{\Higgs}$
coming from Higgs bundles (Theorem \ref{thm:chains-non-empty}).  In
Section \ref{sec:irreducibility-chains} we extend this to the boundary
case $\alpha=\alpha_{\Higgs}$ (Theorem
\ref{thm:non-empty-irreducible-alpha-higgs}). In Section
\ref{sec:application_U_pq} we apply these results to prove
connectedness of the $\U(p,q)$-character varieties. The application to
the components of the global nilpotent cone is then explained in the
last section.

After this work was completed Tristan Bozec managed to give a complete combinatorial description of the irreducible components of the global nilpotent cone in \cite{arXiv:1712.07362} using closely related methods. 

Many of the ideas in this article have their origin in Nigel Hitchin's
work. Indeed, in his seminal paper \cite{hitchin:1987} he introduced
Higgs bundles and their use in the study of character varieties for
surface groups --- in particular for representations in
$\SL_2\bR = \mathsf{SU}(1,1)$ --- and he used length two chains of
line bundles to study
the topology of the moduli spaces. We would like to take this
opportunity to acknowledge his influence on us and our work, and we
take great pleasure in dedicating this article to him.


\section{Recollection on moduli spaces of chains}\label{sec:setup}

By \cite[Section 4]{simpson:1992} the moduli spaces of Higgs bundles carry a $\mathbb{C}^*$-action for which the
fixed points parametrize systems of Hodge bundles.  In the case of $\U(p,q)$-Higgs bundles, these fixed points have an elementary description as moduli spaces of holomorphic chains that were studied in detail in \cite{AGPS}. 

In this section we therefore briefly recall the basic notions on chains of vector bundles and stability for these objects as well as the foundational results on their moduli that we will need. 

\subsection{Definitions and Notation}
Throughout we will fix a smooth projective curve $C$ of genus $g$ over a field $k$ of
characteristic $0$ and assume that $g\geq 1$. We denote the sheaf of differentials on $C$ by $\Omega$.
\begin{remark}
The restriction on the base field enters only through the use of \cite[Proposition 4.5]{AGPS}), for which
no characteristic free proof has been found so far. The assumption on the genus of the curve is needed to ensure that $2g-2\geq 0$, so that \cite[Proposition 4]{GPH} applies).
\end{remark}

\begin{definition}
 A {\em chain} $\cE_\bullet$ of length $r$ on $C$ is a collection
\begin{displaymath}
  \cE_r\xra{\phi_r}\cE_{r-1}\xra{\phi_{r-1}}\dots\xra{\phi_1}\cE_0
\end{displaymath}
consisting of vector bundles $\cE_j$ and morphisms of $\cO_C$-modules $\phi_j\colon \cE_j \to \cE_{j-1}$.
\end{definition}

\begin{definition}
For $\alpha=(\alpha_1,\dots,\alpha_r)\in\bR^{r}$ we will denote the $\alpha$-slope by

$$ \mu_\alpha(\cE_\bullet) := \frac{\sum_{i=0}^r \deg(\mathcal{E})_i +\sum_{i=1}^r \alpha_i \rank(\mathcal{E})_i}{\sum_{i=0}^r  \rank(\mathcal{E})_i}.$$

Where convenient we denote by $\un{n} =\un{\rk}(\cE_\bullet) \in
\bN_0^{r+1}, \un{d} =\un{\deg}(\cE_\bullet) \in \bZ^{r+1}$ the rank
and degree of the chain, and abbreviate $|n|:=\sum_{i=0}^r n_i$. We
will always assume $|n|>0$ (i.e.\ we allow some $n_i=0$, but not all
of them).  To emphasize that the $\alpha$-slope depends only on the
numerical invariants we will also use the notation

$$ \mu_\alpha(\un{n},\un{d}) := \frac{\sum_{i=0}^r d_i +\sum_{i=1}^r \alpha_i n_i}{|n|}.$$
\end{definition}

\begin{remark}In some formulae it is convenient to expand $\alpha$ to include $\alpha_0$, i.e.\ to take $\alpha=(\alpha_0, \alpha_1,\dots,\alpha_r)$. Notice that if $\alpha+\un{c}= (\alpha_0+c, \alpha_1+c,\dots,\alpha_r+c)$ then $\mu_{\alpha+\un{c}}(\cE_\bullet)=\mu_{\alpha}(\cE_\bullet)+c$. Since all important formulae involving $\alpha$-slopes (e.g. \eqref{slopeineq}) involve differences of slopes, we can without loss of generality assume that $\alpha_0=0$.
\end{remark}

\begin{definition} We write $\alpha_{\Higgs}=(i(2g-2))_{i=1,\dots r}$ and will abbreviate by 

$$\alpha>\alpha_{\Higgs}$$ 

\noindent the condition $\alpha_i-\alpha_{i-1} > 2g-2$ for  $i=1,\dots r$ (where we assume $\alpha_0:=0$).

\end{definition}

\begin{definition}

A chain $\cE_\bullet$  is called $\alpha$-semistable, if for all subchains $0\ne \mathcal{F}_\bullet\subset\cE_\bullet$  we have

\begin{equation}\label{slopeineq}
\mu_\alpha(\mathcal{F}_\bullet)\le\mu_\alpha(\cE_\bullet) 
\end{equation}

The chain is called $\alpha$-stable if the inequality is strict for all subchains
\end{definition}

\begin{definition}The stack of chains of rank $\un{n}$ and degree $\un{d}$ is denoted by $\Chain_{\un{n}}^{\un{d}}$. The open substack of $\un{\alpha}$-semistable chains is denoted by $\Chain_{\un{n}}^{\un{d},\alpha-\sst}$.

\end{definition}

\begin{remark}
	Chains of length $1$ are usually called triples $\cE_1 \map{\phi}\cE_0$. For these objects the stability condition depends on a single real parameter $\alpha$. The corresponding $\alpha$-slope is 
	$$  \mu_\alpha(\un{n},\un{d}) := \frac{d_0 + d_1 + \alpha  n_1}{n_0+n_1}.$$
\end{remark}

\begin{remark}\label{rem:Higgs}
	The relation to Higgs bundles is as follows: Any chain $\cE_\bullet$ defines a Higgs bundle $\cE:= \oplus_i (\cE_i \tensor \Omega^{-i})$ with Higgs field $\theta:= \oplus \phi_i \colon \cE \to \cE\tensor \Omega$. Conversely, any Higgs bundle $(\cE,\theta)$ that is a fixed point for the $\bC^*$-action $\theta \mapsto t\theta$ is of this form.
\end{remark}

\subsection{Necessary conditions for the existence of semistable chains}

Our starting point are the conditions for the existence of semistable chains introduced in  \cite{GPH}, which we will need to recall.

\begin{definition}
For a chain $\cE_\bullet =(\cE_r \to \dots \to \cE_0)$ we will call the following chains {\em standard subchains} of $\cE_\bullet$: 
\begin{enumerate}
	\item For $0\leq k <r$ the chain $\cE^{\prime,\geq k}_\bullet:=(0 \to \dots \to 0 \to \cE_k \to \dots \to \cE_0)$.
	\item For all $0\leq k < j \leq r$ such that $n_j <\min\{n_{k},\dots,n_{j-1}\}$ the chain
	$$\cE^{\prime,[k,j]}_\bullet:= (\cE_r \to \dots \to \cE_{j+1} \to \cE_j = \dots = \cE_j \to \cE_{k-1} \to \dots \to \cE_0).$$
\end{enumerate}
 
Dually we will call {\em standard quotients} of $\cE_\bullet$ the chains
\begin{enumerate}
	\item[(3)] For $0< k \leq r$ the chain $\cE^{\pprime,\leq k}_\bullet:=(\cE_r \to \dots \to \cE_k \to 0 \to \dots \to 0)$.
	\item[(4)] For all $0\leq k < j \leq r$ such that $n_k < \min\{n_{k+1},\dots,n_{j}\}$ the chain
	$$\cE^{\pprime,[k,j]}_\bullet:= (\cE_r \to \dots \to \cE_{j+1} \to \cE_k = \dots = \cE_k \to \cE_{k-1} \to \dots \to \cE_0).$$
\end{enumerate}

	Note that these types of chains are dual to each other, i.e.\ a standard subchain of $\cE_\bullet$ defines a standard quotient of the dual chain $\cE_\bullet^\vee$. 
	
\end{definition}

\begin{remark}\label{rem:standard}
As explained in \cite{GPH} this is a slight abuse of notation, because although the chains in (2) come equipped with a morphism to $\cE_\bullet$ this morphism need not be injective. However, it is shown in \cite{GPH} that for stability parameters $\alpha$ satisfying $\alpha_r >\dots > \alpha_0$ for any chain either the canonical map is injective, or there exists a subchain of slope larger than the slope of the standard subchain (constructed from the image of the canonical map). Thus semistable chains of degree $(\un{n},\un{d})$ can exist only if the slopes of the standard subchains satisfy inequality \eqref{slopeineq}, i.e.\ are bounded above by $ \mu_\alpha(\un{n},\un{d})$.  And moreover if equality occurs, the corresponding standard subchain defines an honest subchain of every semistable chain.
\end{remark}

\begin{remark} 
In the above we corrected a typo in \cite[Proposition 4]{GPH}: the given condition on the ranks $n_k$ has to be the one dual to the standard subchains, as the proof is given by duality. Fortunately this does not affect the rest of the arguments in loc.\ cit.\ as the proofs again apply duality.
\end{remark}

With this notation \cite[Proposition 4]{GPH} says that $\alpha$-semistable chains can only exist if none of the standard subchains (resp.\ quotient chains) has $\alpha-$slope $\geq \mu_{\alpha(\un{n},\un{d})}$ (resp. $\leq \mu_{\alpha(\un{n},\un{d})}$). Let us list the explicit form of these necessary conditions :

\begin{itemize}
\item For all $i$ such that $n_i=n_{i-1} $ we have
\begin{equation}\label{cond0}\tag{C0}
 d_i\leq d_{i-1}.
\end{equation}
\item  For all $0\leq k<r$ we have 
\begin{equation}\label{cond1}\tag{C1}
\frac{\sum_{i=0}^k \big(d_i + \alpha_i n_i\big)}{\sum_{i=0}^k n_i} \leq \mu_\alpha(\un{n},\un{d}).
\end{equation}
\item For all $0\leq k < j \leq r$ such that $n_j <\min\{n_{k},\dots,n_{j-1}\}$ we have 
\begin{equation}\label{cond2}\tag{C2}
\frac{\sum_{i\not\in[k,j]}\big(d_i+\alpha_i n_i \big)+(j-k+1)d_j+\big(\sum_{i=k}^j\alpha_i\big)n_j}{\sum_{i\not\in[k,j]}n_i+(j-k+1)n_j}\leq \mu_\alpha(\un{n},\un{d})
\end{equation}
\item For all $0\leq k < j \leq r$ such that $n_k < \min\{n_{k+1},\dots,n_{j}\}$ we have 
\begin{equation}\label{cond3}\tag{C3}
\frac{\sum_{i=k+1}^{j} \big( d_i-d_k +\alpha_i(n_i-n_k) \big)}{\sum_{i=k+1}^j(n_i-n_k)} \leq \mu_\alpha(\un{n},\un{d})
\end{equation}
\end{itemize}

\begin{remark} Note that since $\mu_\alpha$ is linear in $\alpha$, for fixed $\un{n},\un{d}$ the conditions \eqref{cond1}-\eqref{cond3} define half-spaces in the space of stability parameters $\bR^{r+1}$.
\end{remark}

\begin{remark}
We could have subsumed condition \eqref{cond0} in condition \eqref{cond2} by allowing $n_j \leq \min\{n_{k},\dots,n_{j-1}\}$, but we want to stress that \eqref{cond0} is independent of $\alpha$.
\end{remark}
\begin{remark}
 Condition~\eqref{cond3} can equivalently be stated for the
  $\alpha$-slope of the standard quotient $\cE^{'',[k,j]}$ rather than
  for the corresponding subchain. It becomes the following: For all
  $0\leq k < j \leq r$ such that
  $n_k < \min\{n_{k+1},\dots,n_{j}\}$ we have
\begin{equation}\label{cond3-prime}\tag{C3'}
  \frac{\sum_{i\not\in[k,j]}(d_i+\alpha_in_i)+(j-k+1)d_k+\sum_{i=k}^j\alpha_in_k}
  {\sum_{i\not\in[k,j]}n_i+(j-k+1)n_k}
  \geq \mu_\alpha(\un{n},\un{d}).
\end{equation}
\end{remark}

\begin{remark}\label{triplesCs}
	For triples the above conditions reproduce the conditions:
	$$ \alpha_{\min}= \frac{d_0}{n_0} - \frac{d_1}{n_1} \leq
        \alpha \leq \alpha_{\max}
        =\big(1+\frac{n_0+n_1}{|n_0-n_1|}\big)	 \big(\frac{d_0}{n_0}
        - \frac{d_1}{n_1}\big),$$
with $\alpha_{\max} = \infty$ if $n_0=n_1$.
\end{remark}

\section{Sufficiency of the conditions for $\alpha>\alpha_{\Higgs}$}\label{sec:sufficient}

\begin{definition}\hfill
\begin{enumerate}
\item For given $\un{n},\un{d}$ satisfying condition \eqref{cond0} we define the convex stability region: 
\[\mathrm{Stability}_{\un{n}}^{\un{d}} := \{ \alpha \in \bR^{r} | \alpha>\alpha_{\Higgs} \text{ and } \alpha \text{ satisfies conditions \eqref{cond1},\eqref{cond2},\eqref{cond3}} \}.\]

\item The parameter $\alpha\in \mathrm{Stability}_{\un{n}}^{\un{d}} $ is called a {\it critical value} if there exists $\un{n'},\un{d'}$ with $\un{n'}<\un{n}$ such that $\mu_{\alpha}(\un{n'},\un{d'})=\mu_{\alpha}(\un{n},\un{d})$ but $\mu_{\beta}(\un{n'},\un{d'})\ne\mu_{\alpha}(\un{n},\un{d})$ for some $\beta\in \bR^{r}$. If this happens, the condition $\mu_{\alpha}(\un{n'},\un{d'})=\mu_{\alpha}(\un{n},\un{d})$ cuts out a hyperplane called a {\it wall} in the space of stability parameters.

\end{enumerate}
\end{definition}

We can now state our first result:

\begin{theorem}\label{thm:chains-non-empty}
	Assume that $g\geq 1$ and let $\un{n}\in \bN^{r+1},\un{d}\in \bZ^{r+1}$ and $\alpha\in \bR^{r}_{>\alpha_{\Higgs}}$.
	Then the stack $\Chain_{\un{n}}^{\un{d},\alpha-\sst}$ is irreducible and non-empty if and only if $\un{n},\un{d}$ satisfy condition \eqref{cond0} and $\alpha\in \mathrm{Stability}_{\un{n}}^{\un{d}}$.
\end{theorem}

\begin{proof}			
Since we know that the conditions are necessary we only have to show that the conditions are also sufficient. We will argue by induction on $r$ and $|n|=\sum_{i=0}^r n_i$. 

For $r=0$ the conditions are empty and the stack of semi stable bundles of rank $n_0$ and degree $d_0$ is known to be irreducible and nonempty for $g>0$.

For $r>0$ we proceed as in the proof of \cite[Proposition 3.8]{H}, i.e., the basic strategy will be to vary the stability parameter $\alpha$ in order to reduce to the case that $\alpha$ lies on a boundary of $\Stab_{\un{n}}^{\un{d}}$. In this case any semistable chain $\cE_\bullet$ will contain a standard subchain of the same slope so it is properly semistable and this will allow us to conclude by induction, because of the following lemma:

\begin{lemma}\label{combinatorialLemma}
	Let $\un{n} \in \bN^{r+1},\un{d}\in \bZ^{r+1}$ be such that they satisfy condition \eqref{cond0} and suppose that $\alpha \in \Stab_{\un{n}}^{\un{d}}$ lies on a wall defined by equality in one of the conditions  \eqref{cond1},\eqref{cond2},\eqref{cond3}. Let $\un{n}^\prime,\un{d}^\prime$ (resp.  $\un{n}^\pprime,\un{d}^\pprime$) denote the numerical invariants of the standard subchain (resp. the standard quotient chain) defining the wall and set  $(\un{n}^\pprime,\un{d}^\pprime) :=(\un{n},\un{d})-(\un{n}^\prime,\un{d}^\prime)$ (resp.$(\un{n}^\prime,\un{d}^\prime) :=(\un{n},\un{d})-(\un{n}^\pprime,\un{d}^\pprime)$). 
	
	Then $(\un{n}^\prime,\un{d}^\prime)$ and $(\un{n}^\pprime,\un{d}^\pprime)$ satisfy \eqref{cond0} and moreover:
	$$\alpha \in \Stab_{\un{n}^\prime}^{\un{d}^\prime} \text{ and } \alpha \in \Stab_{\un{n}^\pprime}^{\un{d}^\pprime}.$$
\end{lemma}

\begin{remark}
	In case of triples, i.e., for $r=1$ this lemma is easy to see, because in this case the only standard subchains are $0\to \cE_0$ and $\cE_1 \to \cE_1$ (if $n_1<n_0$) and the standard quotient is $\cE_0 \to \cE_0$ (if $n_0<n_1$). In all of these cases $\Stab_{\un{n}^\prime}^{\un{d}^\prime}$ and  $\Stab_{\un{n}^\pprime}^{\un{d}^\pprime}$ are only defined by condition \eqref{cond0}, which is then trivially verified.  
\end{remark}

\begin{proof}[Proof of Lemma \ref{combinatorialLemma}:] This is an elementary verification:
Condition \eqref{cond0} is easily seen for the standard subchains and quotient chains, because the ranks and degrees of the bundles are a subset of those of the original chain and also for a  quotient by a standard subchain the ranks $n_i^\pprime$ are either $0$ or, we have  $$(n^\pprime_j,d^\pprime_j),(n_{j+1}^\pprime,d_{j+1}^\pprime)= ((n_j-n_i,d^\pprime_j-d_i),(n_{j+1}^\pprime-n_i,d_{j+1}^\pprime-d_i)).$$

Let us abbreviate for $0\leq k\leq l\leq r$
$$ \un{n}_{[k,l]}:=(0,\dots,0,n_k,\dots,n_l,0,\dots 0).$$

Suppose that $\alpha$ lies on a wall defined by \eqref{cond1}, i.e.
$$ \mu_{\alpha}(\un{n}_{[0,k]},\un{d}_{[0,k]})=\mu_{\alpha}(\un{n},\un{d}).$$
for some $k$. 

Thus we also have $\mu_{\alpha}(\un{n}_{[k+1,r]},\un{d}_{[k+1,r]})= \mu_{\alpha}(\un{n},\un{d})$.

Then \eqref{cond1} for $(\un{n}^\prime,\un{d}^\prime)=(\un{n}_{[0,k]},\un{d}_{[0,k]})$ is implied by \eqref{cond1} for $\un{n},\un{d}$ and the above equality. Similarly if the condition failed for $\un{n}^\pprime,\un{d}^\pprime$ for some $k^\pprime$, i.e.,
$$\mu_{\alpha}(\un{n}_{[k+1,k^\pprime,]},\un{d}_{[k+1,k^{\pprime}]}) > \mu_{\alpha}(\un{n},\un{d})$$ then since the slope of a sum is a convex combination of the slope of the summands we also have that
$$  \mu_{\alpha}(\un{n}_{[0,k^\pprime,]},\un{d}_{[0,k^{\pprime}]}) > \mu_{\alpha}(\un{n},\un{d}),$$
contradicting \eqref{cond1} for $\un{n},\un{d}$.

The same argument shows that \eqref{cond2} is preserved for $(\un{n}^\prime,\un{d}^\prime)=(\un{n}_{[0,k]},\un{d}_{[0,k]})$, because the union of a subchain of type 2 of invariants $(\un{n}_{[0,k]},\un{d}_{[0,k]})$ with  $(\un{n}_{[k+1,r]},\un{d}_{[k+1,r]})$ will define a subchain of type 2 of  $(\un{n},\un{d})$ and the same applies to the standard quotients of type (3). The same argument gives  the result for  $(n^\pprime,d^\pprime)=(\un{n}_{[k+1,r]},\un{d}_{[k+1,r]})$.

Next suppose that $\alpha$ lies on a wall defined by \eqref{cond2} for some $k<j$. Let us denote by   $(\un{n}^\prime,\un{d}^\prime)=: (\un{n}^{[j,k]},\un{d}^{[j,k]})$ the corresponding numerical invariants.  

A standard subchain of type 2 of a chain with invariants $(\un{n}^\prime,\un{d}^\prime)$ is also a subchain of the same type for $(\un{n},\un{d})$. Similarly, the union of a subchain of type 2 of $(\un{n}^\pprime,\un{d}^\pprime)$ with $(\un{n}^\prime,\un{d}^\prime)$ is also a standard subchain of type 2 of $(\un{n},\un{d})$, so \eqref{cond2} will again be satisfied. 

Given a standard subchain of type 1 of $(\un{n}^\prime,\un{d}^\prime)$ for some index $k^\prime$ is handled by the same argument if $k^\prime \not \in [j,k]$. Suppose that $k^\prime \in [j,k]$. By assumption we then know $\mu(\un{n}^{[j,k^\prime]},\un{d}^{[j,k^\prime]}) \leq \mu_{\alpha}(\un{n},\un{d})$. But this implies that 
$$\mu_{\alpha}(\un{n}^{[j,k^\prime]}- \un{n}^{[j,k]},\un{d}^{[j,k^\prime]} \un{d}^{[j,k]}) \geq \mu_{\alpha}(\un{n},\un{d})\ .$$ 
Since we know that  $\mu_{\alpha}(\un{n}_{[0,k^\prime]},\un{d}_{[0,k^\prime]}) \leq \mu(\un{n},\un{d})$ we must also have 
$$\mu_{\alpha}(\un{n}^\prime_{[0,k^\prime]},\un{d}^\prime_{[0,k^\prime]}) \leq \mu(\un{n},\un{d})\ .$$

For the quotient  $\un{n}^\pprime,\un{d}^\pprime$ the argument is easier, since the union of a subchain of type $1$ with $\un{n}^\prime,\un{d}^\prime$ would be another standard subchain of type 2 of $\un{n},\un{d}$.

The argument for standard quotients is of type 3 is analogous. 

Finally, the case of $\alpha$ lying on a wall defined by \eqref{cond3} is dual to the above.
\end{proof}

Let us now finish the proof of Theorem \ref{thm:chains-non-empty}.
  
{\em Case 1 (Chains of constant rank):} 
In case $\un{n}=(n,\dots,n)$ is constant only condition \eqref{cond1} appears. If $\alpha$ does not lie on the boundary of $\mathrm{Stability}_{\un{n}}^{\un{d}}$ we can apply \cite[Lemma 3.4 and Proposition 3.8 (1)]{H} to find that $\Chain_{\un{n}}^{\un{d},\alpha-\sst}$ is irreducible and that it is non-empty if and only if $\Chain_{\un{n}}^{\un{d},\alpha_\infty-\sst}$ is non-empty for any $\alpha_\infty$ in  $\mathrm{Stability}_{\un{n}}^{\un{d}}$ satisfying $\alpha_{\infty,i}-\alpha_{\infty,i-1} > d_{i-1}-d_i$.
 
In this case any $\alpha_\infty$-unstable chain $\cE_\bullet^\prime$ contains a maximal destabilizing subchain $\cE_\bullet^\prime$ which is again of constant rank by \cite[Proposition 6.9]{GPHS}. 
Since 
$$\mu_{\alpha_\infty}(\cE_\bullet^\prime) = \sum \frac{1}{r+1} (\mu(\cE_i^\prime) + \alpha_{\infty,i})$$
and the same holds for $\mu_{\alpha_\infty}(\cE_\bullet)$ we see that for at least one $i$ the subsheaf $\cE_i^\prime \subset \cE_i$ must be destabilizing.

The stack $\Chain_{\un{n}}^{\un{d},\text{inj}}$ of those chains for which all maps $\cE_i \to \cE_{i-1}$ are injective is smooth and irreducible and the forgetful map to any of the $\Bun_n^{d_i}$ is a smooth fibration. Thus there is an open dense subset of chains such that all of the $\cE_i$ are semistable. Therefore we have shown that in this case  $\Chain_{\un{n}}^{\un{d},\alpha_\infty-\sst}$ is not empty.

If $\alpha$ lies on a wall defined by condition \eqref{cond1}, say for the integer $k$, we can argue by induction on $r$: In this case any $\alpha$ semistable chain has the subchain $\cE_\bullet^\prime =(0 \to \dots \to \cE_k \to \dots \to \cE_0)$ with $\mu_{\alpha}(\cE_\bullet^\prime)= \mu_\alpha(\un{n},\un{d})$ and this subchain is again $\alpha$-semistable. 

Consider the canonical map
$$
\pi\colon \Chain_{\un{n}}^{\un{d},\alpha-\sst} \to
\Chain_{(0,\dots,0,n,\dots,n)}^{(0,\dots,0,d_k,\dots,d_0),\alpha-\sst} 
\times
\Chain_{(n,\dots,n,0,\dots,0)}^{(d_r,\dots,d_{k+1},0,\dots,0),\alpha-\sst}
$$

By induction we know that the factors of the product are non-empty and irreducible and moreover for any pair of chains $(\cE_\bullet^\prime,\cE_\bullet^\pprime)$ in this product we have that
$$\Ext^2(\cE_\bullet^\pprime,\cE_\bullet^\prime)=0$$ 
by \cite[Lemma 4.5]{GPHS}.
Thus the map $\pi$ is a generalized vector bundle and therefore $\Chain_{\un{n}}^{\un{d},\alpha-\sst}$ is irreducible and non-empty also in this case.
This proves the theorem for chains of constant rank.

{\em Case 2 (Chains of non-constant rank)}

We may now assume that the rank $\un{n}$ is not constant.  Again, if $\alpha$ lies in the interior of $\mathrm{Stability}_{\un{n}}^{\un{d}}$ we argue again as in \cite[Proposition 3.8]{H}:

In this case we can find a piecewise linear path $\gamma$ inside of the convex region such that $\gamma(0)=\alpha$ and $\alpha^\prime=\gamma(1)$ lies on a single wall defined by one of the inequalities. By \cite[Lemma 3.4]{H} this reduces the claim to proving that $\Chain_{\un{n}}^{\un{d},\alpha^\prime-\sst}$ is non empty and irreducible if $\alpha^\prime$ lies on a wall defined by one of the conditions \eqref{cond1},\eqref{cond2},\eqref{cond3}.

Thus we may assume that $\alpha=\alpha^\prime$ lies on at least one of the walls.
If it lies on several of the walls choose one of these.
For any $\alpha$-semistable chain $\cE_\bullet$ the corresponding standard subchain $\cE^\prime_\bullet$ (or standard quotient chain) will actually be a subchain of the same $\alpha$-slope. Taking the associated grade chain we get a morphism 
$$\pi\colon \Chain_{\un{n}}^{\un{d},\alpha-\sst} \to  \Chain_{\un{n}^\prime}^{\un{d}^\prime,\alpha-\sst} \times  \Chain_{\un{n}-\un{n}^\prime}^{\un{d}-\un{d}^\prime,\alpha-\sst}.$$
Note that since $\alpha>\alpha_{\Higgs}$ it follows (see \cite[Lemma 4.6]{GPHS}) that $\Ext^2(\cE_\bullet^\pprime, \cE_\bullet^\prime)=0$ for all $(\cE_\bullet^\prime,\cE_\bullet^\pprime) \in  \Chain_{\un{n}^\prime}^{\un{d}^\prime,\alpha-\sst} \times  \Chain_{\un{n}-\un{n}^\prime}^{\un{d}-\un{d}^\prime,\alpha-\sst}$ and hence $\pi$ is a generalized vector bundle. Therefore, again we can conclude by induction because from Lemma \ref{combinatorialLemma} we know that $\alpha^\prime$ also satisfies the inequalities for $(\un{n}^\prime,\un{d}^\prime)$ and $(\un{n}^\pprime,\un{d}^\pprime)$. 
\end{proof}

\section{Irreducibility of moduli spaces of semistable chains
  (non-coprime case)}
\label{sec:irreducibility-chains}

\begin{theorem}
  \label{thm:non-empty-irreducible-alpha-higgs}
	Assume $g>1$. Let $\un{n}\in \bN^{r+1},\un{d}\in \bZ^{r+1}$ and $\alpha= \alpha_{\Higgs}$.
	Then the coarse moduli space $M_{\un{n}}^{\un{d}}$ of $\Chain_{\un{n}}^{\un{d},\alpha_{\Higgs}-\sst}$ is irreducible and non-empty if and only if $\un{n},\un{d}$ satisfies condition \eqref{cond0} and $\alpha_{\Higgs}\in \overline{\mathrm{Stability}}_{\un{n}}^{\un{d}}$.
\end{theorem}
Note that in contrast to the previous section we only prove the result for the coarse moduli space. 
 
\begin{proof}
As before we will prove the statement by induction on $r$ and $|n|$. For $r=0$ the claim follows from the irreducibility of the stack of semistable bundles.
	
If $\alpha_{\Higgs}$ lies on one of the walls defined by \eqref{cond1},\eqref{cond2},\eqref{cond3}	then by Remark \ref{rem:standard} any semistable chain will contain a standard subchain of the same slope, say of rank $\un{n}^\prime, \un{d}^\prime$. 
This defines a surjective morphism
$$p \colon \Chain_{\un{n}}^{\un{d},\alpha_{\Higgs}-\sst} \to \Chain_{\un{n}^\prime}^{\un{d}^\prime,\alpha_{\Higgs}-\sst}\times \Chain_{\un{n}-\un{n}^\prime}^{\un{d}-\un{d}^\prime,\alpha_{\Higgs}-\sst} $$
and all its fibers consist of bundles in the same S-equivalence class. Thus in this case we can conclude by induction.

If $\alpha_{\Higgs}$ does not lie on any wall then $\alpha_{\Higgs}-$semistability is equivalent to $\alpha$-stability for $\alpha$ in some neighborhood of $\alpha_{\Higgs}$, so that by Theorem \ref{thm:chains-non-empty} we already know that the stack of semistable chains is irreducible and therefore the same holds for its coarse moduli space.

Finally suppose that $\alpha_{\Higgs}$ lies on some other wall in the space of stability parameters.
Choose a short line segment $[\alpha_-,\alpha_+]$ through $\alpha$ such that $\alpha$ is the only critical value in the line and such that $\alpha_+\in \mathrm{Stability}_{\un{n}}^{\un{d}}$. Let us denote by $\Chain_{\un{n}}^{\un{d},\alpha_{+},HN-t}$ the Harder--Narasimhan stratum of type $t$ in the moduli stack of chains.

Then we know that
\begin{align*} 
\Chain_{\un{n}}^{\un{d},\alpha_{\Higgs}-\sst} &= \Chain_{\un{n}}^{\un{d},\alpha_{+}-\sst} \cup \bigcup_{t\in I} \Chain_{\un{n}}^{\un{d},\alpha_{+},HN-t} \\
\end{align*}
where $I$ is the set of types of Harder--Narasimhan-filtrations for $\alpha_+$ such that the $\alpha$-slopes of the graded quotients are all equal to $\mu_\alpha(\un{n},\un{d})$.

Also recall that Harder--Narasimhan-strata have a partial ordering defined by $$t^\prime < t \Leftrightarrow \overline{\Chain_{\un{n}}^{\un{d},\alpha_{+},HN-t}} \cap \Chain_{\un{n}}^{\un{d},\alpha_{+},HN-t^\prime} \neq \emptyset \text{ and } t\neq t^\prime.$$ 

In this situation it seems difficult to argue by dimension reasons as in \cite[Section 3]{H}, because the stacks may now be singular. However, we claim that any $\alpha$-semistable chain is S-equivalent to a chain in $\Chain_{\un{n}}^{\un{d},\alpha_{+}-\sst}$. This will imply that the morphism 

$$\Chain_{\un{n}}^{\un{d},\alpha_{+}-\sst} \to M_{\un{n}}^{\un{d}}$$ 

\noindent is surjective and since $\Chain_{\un{n}}^{\un{d},\alpha_{+}-\sst}$ is irreducible by Theorem \ref{thm:chains-non-empty} we can then conclude the same for $M_{\un{n}}^{\un{d}}$.

Suppose then that $\cE_\bullet \in \Chain_{\un{n}}^{\un{d},\alpha_{+},HN-t}\subset \Chain_{\un{n}}^{\un{d},\alpha_{\Higgs}-\sst}$. Let 
$$0=\cF^0\subset \cF_\bullet^1 \subset \dots  \subset \cF_\bullet^k = \cE_\bullet$$
\noindent be the $\alpha_+$-Harder--Narasimhan-filtration and $\cE_\bullet^i := \cF_\bullet^i/\cF_\bullet^{i-1}$ the subquotients. Then $\cE$ is S-equivalent to $\oplus \cE_\bullet^i$. Consider the stack of extensions in the opposite direction 
$$\un{\Ext}(\cE_\bullet^1,\dots,\cE_\bullet^k) = \langle (\cG^k_\bullet \subset \dots \subset \cG^1_\bullet, \psi_i\colon \cG_\bullet^i/\cG_\bullet^{i+1} \cong \cE_i ) \rangle.$$

\noindent Any such extension can be degenerated into the split extension. It is therefore $\alpha$-semistable and can moreover only lie in Harder--Narasimhan-strata of type $t^\prime$ with $t^\prime \geq t$. Finally if such an extension lies in the same Harder--Narasimhan-stratum as $\cE_\bullet$  then the extension must be a split extension. (This is by induction, e.g. the maximal destabilizing $\cF^\prime_\bullet \to \cG_\bullet^1$ satisfies that the composition $\cF^\prime_\bullet \to \cG_\bullet^1 \to \cG_\bullet^1/\cG_{\bullet}^{2}$ is an isomorphism, because the kernel has slope $\geq \mu_{\alpha_+}(\cF^\prime_\bullet)$ and $\cG_\bullet^{2}$ has a filtration by bundles of smaller slope, so the kernel must be $0$.)

We are therefore reduced to show that there exists a non-trivial extension, because this would imply that $\cE_\bullet$ is S-equivalent to a chain contained in a less unstable Harder--Narasimhan-stratum and by induction this implies that $\cE_\bullet$ is S-equivalent to a $\alpha_+$-semistable chain.

The existence of a non-trivial extension follows from \cite[Proposition 4.5]{AGPS} and \cite[Section 3]{H}. Let us recall this in the special case we need. Given chains $\cE_\bullet^\prime,\cE_\bullet^\pprime$  we will denote 
$$\cHom^\bullet(\cE_\bullet^\pprime,\cE_\bullet^\prime) := [\oplus \cHom(\cE_i^\pprime,\cE_i^\prime) \map{b} \oplus \cHom(\cE_i^\pprime,\cE_{i-1}^\prime)]$$
the complex of sheaves that computes the $\bR \cHom$ in the category of chains. The differential $b$ is the sum of the commutators defined by the chain maps.
We will abbreviate 
$$\chi(\cE_\bullet^\pprime,\cE_\bullet^\prime):= \chi (\bH^*(C,\cHom^\bullet(\cE_\bullet^\pprime,\cE_\bullet^\prime)))$$

\begin{theorem}[{\cite[Proposition 4.5]{AGPS}}]\label{chi}
  Let $\cE_\bullet^\prime,\cE_\bullet^\pprime$ be
  $\alpha_{\Higgs}$-semistable chains with
  $\mu_{\alpha_{\Higgs}}(\cE_\bullet^\prime)
  =\mu_{\alpha_{\Higgs}}(\cE_\bullet^\pprime)$.
  Then we have
  $$ \chi(\cE_\bullet^\pprime,\cE_\bullet^\prime) \leq 0$$ and $=$
  holds if and only if the differential $b$ of the complex is an
  isomorphism.
\end{theorem}
\begin{remark}
	For the case $r=1$ this result already appears in \cite[Proposition 4.7]{bradlow-garcia-prada-gothen:2004}
\end{remark}

This version stated in the given reference states a slightly weaker result, which however implies the above: As remarked in \cite[Corollary 3.5]{H}, we have replaced the condition of $b$ not being generically an isomorphism by $b$ not being an isomorphism, because if $b$ is generically an isomorphism then the complex is quasi-isomorphic to a complex that is a concentrated in degree $1$ and is a torsion sheaf in this degree. In this case the Euler characteristic is also $<0$. 

Moreover one can also replace the condition that the chains are polystable by the condition that they are semistable, because for chains that are extensions of polystable chains the corresponding $\cHom$ complexes also admit a filtration with graded quotients equal to the ones of the polystable chains and $\chi$ is additive with respect to these filtrations.

Thus we can conclude that $\Ext^1_{\Chain}(\cE_\bullet^1,\dots,\cE_\bullet^k) \neq (0)$ unless for all $i<j$ the complex 
$\cHom^\bullet(\cE_\bullet^i,\cE_\bullet^{j}) \sim 0$ (i.e. is quasi-isomorphic to $0$). This implies that
$\cHom^\bullet(\cE_\bullet^1,\oplus_{j>1} \cE_\bullet^{j}) \sim 0$. Now if the sets 
$$\supp(\rk^1) := \{ i | \rk(\cE_i^1) \neq 0 \} \text{ and } \supp(\rk^{2,\dots k}):= \{ i | \rk( \oplus_{j>1} \cE_i^j ) \neq 0 \}$$
are strings of consecutive integers  we can conclude from \cite[Corollary 3.7]{H} that the numerical invariants of $\oplus_{j>1} \cE^j_\bullet$ are the invariants of a standard subchain of $\cE_\bullet$. Since we already excluded the case that $\alpha_{\Higgs}$ lies on the corresponding wall this cannot happen.

Finally if the sets $\supp(\rk^1), \supp(\rk^{2,\dots k})$ are not connected strings of integers, then the corresponding chains will be  sums of chains supported in different intervals we will write this type of summands as $\cE_\bullet^1 = \cE_\bullet^{1\prime} \cup \cE_\bullet^{1\pprime}$. By semistability each of the summands is semistable of the same slope and there are no extensions between the summands.   Thus if $\cHom^\bullet(\cE_\bullet^1,\oplus_{j>1} \cE_\bullet^{j}) \sim 0$ and $\cE_\bullet^1 =  \cE_\bullet^{1\prime} \cup \cE_\bullet^{1\pprime}$ then also
$\cHom^\bullet(\cE_\bullet^{1\prime},\cE_\bullet^{1\pprime}\oplus_{j>1} \cE_\bullet^{j}) \sim 0$ and similarly for $\cE_\bullet^{1\pprime}\oplus_{j>1} $.  After reordering the summands we can therefore apply the corollary and find a standard subchain that gives a subchain of the same slope and we again find a contradiction.
\end{proof}

\subsection{Triples: the case $r=1$} In view of our main application, we briefly describe the special case $r=1$, in which the chains are of the form $\cE_1 \to \cE_0$.  In this case the conditions \eqref{cond1}-\eqref{cond3} reduce to the inequality in Remark \ref{triplesCs}, $\alpha_{\Higgs}=2g-2$, and Theorem \ref{thm:non-empty-irreducible-alpha-higgs} becomes the statement that

\begin{theorem}
  \label{thm:r=1}
	Assume $g>1$. Let $\un{n}\in \bN^{2},\un{d}\in \bZ^{2}$ and $\alpha= 2g-2$. Then the coarse moduli space $M_{\un{n}}^{\un{d}}$ of $\Chain_{\un{n}}^{\un{d},\alpha_{2g-2}-\sst}$ is irreducible and non-empty if and only if $\un{n},\un{d}$ satisfies condition \eqref{cond0} and $\alpha_{\min}\le 2g-2\le\alpha_{\max}$ where  $\alpha_{\min}$ and $\alpha_{\max}$ are as in Remark \ref{triplesCs}.
\end{theorem}

\noindent In a similar way to Theorem \ref{thm:non-empty-irreducible-alpha-higgs},  Theorem \ref{thm:r=1} can be proved by treating separately three different possibilities for the location of $2g-2$ in the interval $[\alpha_{\min},\alpha_{\max}]$:
\begin{itemize}
\item If $2g-2=\alpha_{\min}$ then a triple  $\cE_1\xra{\phi_1}\cE_0$ is  $2g-2$-semistable if and only if $\cE_0$ and $\cE_1$ are semistable bundles and $\phi_1=0$. The moduli space $M_{\un{n}}^{\un{d}}$ is thus the product of two non-empty irreducible moduli spaces. 
\item If $2g-2=\alpha_{\max}$ then Theorem 7.7 and 8.15 in \cite{bradlow-garcia-prada-gothen:2004} show that $M_{\un{n}}^{\un{d}}$ is non-empty and irreducible. 
\item If $\alpha_{\min}< 2g-2<\alpha_{\max}$, regardless whether $2g-2$ is itself a critical value, we can pick $\epsilon>0$ such that $\alpha^+=2g-2+\epsilon$ is generic with no critical values between $2g-2$ and $\alpha^+$. We consider the map
\begin{equation}\label{surjmap}
\Chain_{\un{n}}^{\un{d},\alpha_{+}-\sst} \to M_{\un{n}}^{\un{d}}.
\end{equation}
 Our goal is to show that this map is surjective since then the fact that $\Chain_{\un{n}}^{\un{d},\alpha_{+}-\sst}$ is non-empty and irreducible will imply that the same is true for $M_{\un{n}}^{\un{d}}$.  We thus consider a triple, say $T$, representing a point $[T]\in M_{\un{n}}^{\un{d}}$ and suppose that it is not $\alpha_+$-semistable.  Then $T$ has a maximally $\alpha_+$-destabilizing sub-triple $T'\subset T$; indeed $T'$ is the first term in the Harder--Narasimhan filtration for $T$.  Assuming that there are non-trivial extensions in which $T/T'$ is a sub-triple and $T'$ the quotient, we can replace $T$ by one such extension, say $\tilde{T}$.  It will follow by construction that $\tilde{T}$ is $\alpha$-semistable and represents the same point as $T$ in $M_{\un{n}}^{\un{d}}$.  Moreover, if $\tilde{T}$ is not $\alpha_+$-semistable then it lies in a ``less unstable'' Harder--Narasimhan-stratum than $T$.  After a finite number of iterations this process produces a representative for $[T]$ which is $\alpha_+$-semistable, thus completing the argument.  The proof thus hinges on the existences of the requisite non-trivial extensions, i.e.\ on the non-vanishing of  $\Ext^1(T', T/T')$.  This is guaranteed by the $r=1$ version of Theorem \ref{chi}, i.e.\ \cite[Proposition 4.7]{bradlow-garcia-prada-gothen:2004}.
\end{itemize}
\section{Connectedness of the moduli space of $\mathsf{U}(p,q)$-Higgs bundles}\label{sec:application_U_pq}

The irreducibility result of the previous section allows us to finally settle the connectedness of the moduli space of $\U(p,q)$-Higgs bundles, completing the results of \cite{bradlow-garcia-prada-gothen:2003} and \cite{bradlow-garcia-prada-gothen:2004}.

Let us start by recalling the setup.
A $\U(p,q)${\em-Higgs bundle} is a collection $$(\cV,\cW,\beta\colon \cW\to \cV\otimes \Omega,\gamma\colon \cV \to \cW\tensor \Omega),$$ where $\cV,\cW$ are vector bundles on $C$ with $\rank(\cV)=p$, $\rank(\cW)=q$. 
 
We will denote by $a=\deg(\cV)$ and $b=\deg(W)$ and denote by $\Higgs_{\U(p,q)}^{a,b}$ the stack of  $\U(p,q)$-Higgs bundles. Any $\U(p,q)$-Higgs bundle defines a Higgs bundle $\cE:= \cV \oplus \cW$ and this induces a notion of semistability for $\U(p,q)$-Higgs bundles. 

We will denote by $\Higgs_{\U(p,q)}^{a,b,\sst}\subseteq \Higgs_{\U(p,q)}^{a,b}$ the open substack of semistable Higgs bundles and by $\cM_{U(p,q)}^{a,b}$ the corresponding coarse moduli space.

By duality we know that $$\Higgs_{\U(p,q)}^{a,b} \cong \Higgs_{\U(p,q)}^{-a,-b}$$ and this isomorphism preserves semistability. We may therefore always assume that $\mu(\cV)\geq\mu(\cW)$.

Let us recall the analytic argument relating these spaces to moduli spaces of triples.
Hitchin showed that there is a proper map (see
\cite[Section~4]{bradlow-garcia-prada-gothen:2003} for details for the
case of $\U(p,q))$-Higgs bundles): 
\begin{align*}
  f\colon\mathcal{M}_{\U(p,q)}^{a,b} &\to \bR,\\
  (\cV,\cW,\beta,\gamma) &\mapsto \norm{\beta}^2+\norm{\gamma}^2.
\end{align*}
Here the norms are $L^2$-norms, taken with respect to hermitian
metrics on $\cV$ and $\cW$ satisfying Hitchin's equations.

Let $\mathcal{N}_{\U(p,q)}^{a,b} \subset \mathcal{M}_{\U(p,q)}^{a,b}$ be the subspace of local minima of $f$.
Then \cite[Theorem 4.6]{bradlow-garcia-prada-gothen:2003} identifies
$\mathcal{N}_{U(p,q)}^{a,b}$ as the subspace of
$(\cV,\cW,\beta,\gamma)$ with $\beta=0$. In turn, there is an isomorphism
\begin{align*}
 \mathcal{N}_{\U(p,q)}^{a,b} &\xra{\simeq} M^{\un{d},\alpha_{\Higgs}}_{\un{n}},\\
  (\cV,\cW,0,\gamma)&\mapsto (\cV\xra{\gamma}\cW\otimes \Omega)
\end{align*}
where $\un{d}=(d_0,d_1)=(b+2g-2,a)$, $\un{n}=(n_0,n_1)=(q,p)$ and
$\alpha_{\Higgs}=2g-2$. In other words, $M^{\un{d},\alpha}_{\un{n}}$ is the
moduli space of $2g-2$-semistable triples.

As $f$ is proper, we know that connectedness of the subspace of local minima
$\mathcal{N}_{\U(p,q)}^{a,b}$ implies connectedness of $\mathcal{M}_{\U(p,q)}^{a,b}$
(see \cite[Proposition~4.2]{bradlow-garcia-prada-gothen:2003}).  The results in  \cite{bradlow-garcia-prada-gothen:2003}) showed only that the restriction of $\mathcal{N}_{\U(p,q)}^{a,b}$ to the stable locus in $\mathcal{M}_{\U(p,q)}^{a,b}$ is connected.  This is sufficient to prove the connectedness of $\overline{\mathcal{M}_{\U(p,q)}^{a,b;\mathrm{stable}}}$ (the closure of the stable locus) but leaves open the possibility of additional connected components in which all points are strictly polystable.  Now, as a consequence of Theorem~\ref{thm:non-empty-irreducible-alpha-higgs} we can rule out this possibility and conclude:

\begin{theorem}
  \label{thm:upq-connected}
  The moduli space $\mathcal{M}_{\U(p,q)}^{a,b}$ is connected.
\end{theorem}

\begin{remark}
   As is common in this circle of problems, one could replace the analytic argument above by the result that the $\bC^*$ on $\mathcal{M}_{\U(p,q)}^{a,b}$ has proper fixed point sets.
\end{remark}

\begin{remark}
  Our results imply that there are no components of the moduli space
  consisting entirely of strictly semistable $\U(p,q)$-Higgs bundles,
  except in the case of rigidity when $p\neq q$ and the Toledo
  invariant is maximal (cf.~\cite{bradlow-garcia-prada-gothen:2003}).
  For such a component, the representations in the corresponding
  components of the character variety all factor through a proper
  subgroup of $\U(p,q)$ and thus do not have Zariski dense image. This
  would have been surprising in view of results of Kim and Pansu
  \cite{kim-pansu} on the deformability of surface group
  representations into representations with Zariski dense image
  (except in the aforementioned case and an analogous one for the
  group $\mathsf{SO}^*(4m+2)$).  The
  Kim--Pansu results apply to representations into all real forms of
  $\mathsf{SL}(n,\mathbb{C}), \mathsf{O}(n,\mathbb{C})$ or
  $\mathsf{Sp}(2n,\mathbb{C})$ but hold only for curves of high genus
  compared to the dimension of the target group. They thus do not
  entirely rule out the possibility of such anomalous components for
  surface group representations of low genus curves.  Our results show
  that if such components exist then, at least in the case of
  $\SU(p,q)$ representations, the explanation is not related to a lack
  of stable points in the corresponding components of the Higgs bundle
  moduli space.
\end{remark}

\section{Application to the irreducible components of the global nilpotent cone}


As noted in the introduction, knowledge about moduli spaces of chains is known to imply results on the irreducible components of the $0$-fiber of Hitchin's fibration. Let us recall this relation. As in \cite{H} we denote by $\Higgs_n^{d,\sst}$ the moduli stack of semistable Higgs bundles of rank $n$ and degree $d$ on $C$, i.e., the stack of pairs $(\cE,\theta)$ where $\cE$ is a vector bundle of rank $n$ and degree $d$ and $\theta \colon \cE \to \cE\tensor\Omega$ a morphism of $\cO_C$-modules. The coarse moduli space of semistable Higgs bundles will be denoted $M_{Dol,n}^d$ and $h\colon M_{Dol,n}^d \to \cA=\oplus H^0(C,\Omega^i)$ is the Hitchin fibration. 

The fiber $h^{-1}(0)$ is called global nilpotent cone. We already recalled that the fixed points for the action of $\bC^*$ on $M_{Dol,n}^d$ are moduli spaces of chains. We denote by $F_{\un{n},\un{d}}\subset h^{-1}(0)$ the corresponding subsets of the moduli of Higgs bundles of the form $\cE=\oplus \cE_i, \theta=\oplus \theta_i$, where $\theta_i\colon \cE_i \to \cE_{i-1}\tensor \Omega$ as in Remark \ref{rem:Higgs}. Finally we denote by $F_{\un{n},\un{d}}^-$ the subschemes of those points $(\cE,\theta)$ for which $\lim_{t\to \infty} (\cE,t\theta) \in F_{\un{n},\un{d}}$. As $h^{-1}(0)$ is projective these strata define a decomposition $h^{-1}(0)=\cup F_{\un{n},\un{d}}^-$. 

If $(n,d)$ are coprime this is the Bia\l ynicki-Birula decomposition
and therefore all of the strata $F_{\un{n},\un{d}}^-$ are smooth and
known to be of dimension $n^2(g-1)+1$. Their closures are thus the irreducible components \cite[Proposition 9.1]{HT}.
	
For general $(n,d)$ the strata $F_{\un{n},\un{d}}^-$ can be singular, so we have to argue more carefully. By \cite{LaumonNilp} and \cite{GinzburgNilp}, the fibers of $h$ when considered in $\Higgs_n^{d,\sst}$ are of pure dimension $n^2(g-1)+1$. This implies that the irreducible components in $h^{-1}(0)$ cannot have larger dimension and since dimension is upper semicontinuous they also have to be of pure dimension $n^2(g-1)+1$.

By \cite[Proposition 2.6]{H} the strata $F_{\un{n},\un{d}}^-$ have a
partial ordering. Namely defining $\wt(\un{n},\un{d}):=  -2\sum_{i<j}
(j-i) n_in_j (\frac{d_j}{n_j}-\frac{d_i}{n_i}).$ we know that for any
$(\un{n},\un{d})$ the subvariety $$\bigcup_{{(\un{m},\un{e}) \atop
    \wt(\un{m},\un{e}) \geq \wt(\un{d},\un{n})}} F_{\un{m},\un{e}}^-$$ is closed. 
Moreover, by \cite[Corollary 2.13]{H} this is also a union of irreducible components of $h^{-1}(0)$.  Therefore for each $(\un{n},\un{d})$ such that $F_{\un{n},\un{d}}\neq \emptyset$ the closure of $F_{\un{n},\un{d}}^-$ constitutes at least one irreducible component of $h^{-1}(0)$. 
			 
\begin{corollary}\label{cor:irredcomp}
	For $g>1$ and any $(n,d)\in \bN \times \bZ$ there are irreducible components of $h^{-1}(0)$ contained in the closure of $F_{\un{n},\un{d}}^-$ if and only if the corresponding numerical invariants of chains satisfy the conditions \eqref{cond0},\eqref{cond1},\eqref{cond2},\eqref{cond3}. If $(n,d)$ are coprime, the $F_{\un{n},\un{d}}^-$ are irreducible.
\end{corollary}

\end{document}